\documentclass{article}
\usepackage{amsfonts}
\usepackage{amsmath}
\usepackage{amssymb}
\usepackage{graphicx}

\newtheorem{theorem}{Theorem}

\newtheorem{corollary}[theorem]{Corollary}

\newenvironment{proof}[1][Proof]{\noindent\textbf{#1.} }{\ \rule{0.5em}{0.5em}}
\input{tcilatex}
\begin{document}

\author{Nazl\i\ Y\i ld\i z \.{I}kikarde\c{s}, Musa Demirci, G\"{o}khan
Soydan, \and \.{I}smail Naci Cang\"{u}l}
\title{THE GROUP STRUCTURE OF BACHET ELLIPTIC CURVES OVER FINITE FIELDS $%
\mathbb{F}_{p}$}
\date{}
\maketitle

\begin{abstract}
Bachet elliptic curves are the curves $y^{2}=x^{3}+a^{3}$ and in this work
the group structure $E(\mathbb{F}_{p})$ of these curves over finite fields $%
\mathbb{F}_{p}$ is considered. It is shown that there are two possible
structures $E(\mathbb{F}_{p})\cong C_{p+1}$ or $E(\mathbb{F}_{p})\cong
C_{n}\times C_{nm},$ \ for $m,n\in 
\mathbb{N}
,$ according to $p\equiv 5~\left( \func{mod}6\right) $ and $p\equiv 1~\left( 
\func{mod}6\right) ,$ respectively. A result of Washington is restated in a
more specific way saying that if $E(\mathbb{F}_{p})\cong Z_{n}\times Z_{n},$
then $p\equiv 7~\left( \func{mod}12\right) $ and $p=n^{2}\mp n+1.$
\end{abstract}

\section{\protect\large Introduction}

$\footnote{\textit{AMS 2000 Subject Classification Number : }11G20, 14H25,
14K15, 14G99
\par
\textit{Keywords: Elliptic curves over finite fields, rational points}}$%
\footnote{%
This work was supported by the research fund of Uludag University project
no: F-2003/63 and F-2004/40}Let $p$ be a prime. We shall consider the
elliptic curves 
\begin{equation}
E:~y^{2}\equiv x^{3}+a^{3}~\left( \func{mod}p\right)
\end{equation}%
where $a$ is an element of $\mathbb{F}_{p}^{\ast }=\mathbb{F}_{p}-\left\{
0\right\} .$ Let us denote the group of the points on $E$ by $E\left( 
\mathbb{F}_{p}\right) .$

If $\mathbb{F}$ is a field, then an elliptic curve over $\mathbb{F}$ has,
after a change of variables, a form%
\begin{equation*}
y^{2}=x^{3}+Ax+B
\end{equation*}%
where $A$ and $B\in \mathbb{F}$ with $4A^{3}+27B^{2}\neq 0$ in $\mathbb{F}.$
Here $D=-16\left( 4A^{3}+27B^{2}\right) $ is called the discriminant of the
curve. Elliptic curves are studied over finite and infinite fields. Here we
take $\mathbb{F}$ to be a finite prime field $\mathbb{F}_{p}$ with
characteristic $p>3.~$Then $A,B\in \mathbb{F}_{p}$ and the set of points $%
\left( x,y\right) \in \mathbb{F}_{p}\times \mathbb{F}_{p},$ together with a 
\textit{point o at infinity} is called the set of $\mathbb{F}_{p}-$ \textit{%
rational points }of $E$ on $\mathbb{F}_{p}$ and is denoted by $E\left( 
\mathbb{F}_{p}\right) .$ $N_{p}$ denotes the number of rational points on
this curve. It must be finite.

In fact one expects to have at most $2p+1$ points (together with \textit{o)(}%
for every $x$, there exist a maximum of $2$ $y^{^{\prime }}$s)$.$ But not
all elements of $\mathbb{F}_{p}$ have square roots. In fact only half of the
elements of $\mathbb{F}_{p}$ have a square root. Therefore the expected
number is about $p+1.$

It is known that%
\begin{equation*}
N_{p}=p+1+\underset{x=0}{\overset{p-1}{\sum }}\chi \left( x^{3}+Ax+B\right) .
\end{equation*}%
Here we use the fact that the number of solutions of $y^{2}\equiv u\ \
\left( p\right) ~$is $1+\chi \left( u\right) .$

The following theorem of Hasse quantifies this result.

\begin{theorem}[Hasse 1922]
$N_{p}$ $<\left( \sqrt{p}+1\right) ^{2}.$
\end{theorem}

Now we look at the algebraic structure of $E\left( \mathbb{F}_{p}\right) .$

Let $P\left( x_{1},y_{1}\right) $ and $Q\left( x_{2},y_{2}\right) $ be two
points on $E:y^{2}=x^{3}+Ax+B.$

Let also%
\begin{equation*}
m=\left\{ 
\begin{array}{c}
\left( y_{2}-y_{1}\right) /\left( x_{2}-x_{1}\right) \ \ \ \ \ \ \ \ \ \ \
if\ \ P\neq Q \\ 
\left( 3x_{1}^{2}+A\right) /2y_{1}\ \ \ \ \ \ \ \ \ \ \ \ \ \ \ \ \ \ \ if\
\ P=Q%
\end{array}%
\right.
\end{equation*}%
where $y_{1}\neq 0$, while when $y_{1}=0$, the point is of order $2$.

\begin{equation*}
x_{3}=m^{2}-x_{1}-x_{2}~\text{and}~y_{3}=m\left( x_{1}-x_{3}\right) -y_{1}
\end{equation*}%
then%
\begin{equation*}
P+Q=\left\{ 
\begin{array}{l}
\text{\textit{o \ \ \ \ \ \ \ \ \ \ \ \ \ \ \ \ \ \ \ \ \ }}%
if~~x_{1}=x_{2}~~\ \ \text{and}~y_{1}+y_{2}=0 \\ 
Q~\ \ \ \ \ \ \ \ \ \ \ \ \ \ \ \ \ \ \ \ \ \ \ \ \ \ \ \ \ \ \ \ if\ \ P=Q\
\  \\ 
\left( x_{3},y_{3}\right) ~~\ \ \ \ \ \ \ \ \ \ \ \ \ \ \ \ \ \ \ \ \ \text{%
otherwise}\ \ \ \ \ \ \ \ \ \ \ \ \ \ \ \ \ 
\end{array}%
.\right.
\end{equation*}%
By definition $-P=\left( x,-y\right) .$\ \ \ \ 

\bigskip Because of the definition of addition in an arbitrary field, it
takes very long to make any addition and the results are very complicated.

Here we shall deal with Bachet elliptic curves $y^{2}=x^{3}+a^{3}$ modulo $%
p. $ Let $N_{p,a}$ denote the number of rational points on this curve. Some
results on these curves have been given in $\cite{Desoca},$ and $\cite{SDYC}%
. $

A historical problem leading to Bachet elliptic curves is that how one can
write an integer as a difference of a square and a cube. In another words,
for a given fixed integer $c$, search for the solutions of the Diophantine
equation $y^{2}-x^{3}=c.$ This equation is widely called as Bachet or
Mordell equation. The existence of duplication formula makes this curve
interesting. This formula was found in 1621 by Bachet. When $(x,y)$ is a
solution to this equation where $x,~y\in 
\mathbb{Q}
$, it is easy to show that $\left( \frac{x^{4}-8cx}{4y^{2}},~\frac{%
-x^{6}-20cx^{3}+8c^{2}}{8y^{3}}\right) $\ is also a solution for the same
equation. Furthermore, if $(x,y)$ is a solution such that $xy\neq 0$ and $%
c\neq 1,~-432$, then this leads to infinitely many solutions, which could
not proven by Bachet. Hence if an integer can be stated as the difference of
a cube and a square, this could be done in infinitely many ways.

If $p\equiv 5~\left( \func{mod}6\right) ,$ it is well known that $E\left( 
\mathbb{F}_{p}\right) \cong C_{p+1},$ the cyclic group of order $p+1,$ $\cite%
{Schmitt}$. But when $p\equiv 1~\left( \func{mod}6\right) $, there is no
result giving the group structure of $E\left( \mathbb{F}_{p}\right) .$ In
this work, we discuss this situation. We show that this group is isomorphic
to a direct product of two cyclic groups $C_{n}$ and $C_{nm},$ i.e. 
\begin{equation*}
E\left( F_{p}\right) \cong C_{n}\times C_{nm}
\end{equation*}%
for $m,n\in 
\mathbb{N}
.$ If we denote the order of $E\left( \mathbb{F}_{p}\right) $ by $N_{p,a}$,
then 
\begin{equation*}
N_{p,a}=n^{2}m=p+1-b
\end{equation*}%
where $b>0$ when $a\in Q_{p},$ and $b<0$ otherwise. Here$\ b$ is the trace
of the Frobenius endomorphism.

\section{Bachet Elliptic curves having a group of the form $C_{n}\times
C_{nm}${\protect\large \thinspace \thinspace }}

Let E\ be the curve in $\left( 1\right) .$ Then its twist is defined as the
curve $y^{2}\equiv x^{3}+g^{3}a^{3},$ where $g$\ is an element of $%
Q_{p}^{^{\prime }}$, the set of quadratic non-residues modulo $p.$ As usual, 
$Q_{p}$ denotes the set of quadratic residues modulo $p.$ Here note that if $%
a\in Q_{p},$ then $ga\in Q_{p}^{^{\prime }}$ and when $a\in Q_{p}^{^{\prime
}},$ then $ga\in Q_{p}.$ It is easy to show that b of $\left( 1\right) \ $%
and of its twist have different signs$.$ Therefore

\begin{theorem}
Let $p\equiv 1~\left( \func{mod}6\right) $ be a prime. If $\left( 1\right) $
has the group isomorphic to $C_{n}\times C_{nm}$ with order $n^{2}m=p+1-b,$
then its twist is isomorphic to $C_{r}\times C_{rs}$ with order $%
r^{2}s=p+1+b.$

Let us define $t=\left\vert b\right\vert .$ That is%
\begin{equation*}
t=\left\vert p+1-N_{p,a}\right\vert .
\end{equation*}

We first have
\end{theorem}

\begin{theorem}
a) Let $p\equiv 1~\left( \func{mod}12\right) $ be a prime. Then 
\begin{equation*}
b\equiv 2~\left( \func{mod}12\right) ~iff~~N_{p,a}\equiv 0~\left( \func{mod}%
12\right)
\end{equation*}%
and%
\begin{equation*}
b\equiv 10~\left( \func{mod}12\right) ~iff~~N_{p,a}\equiv 4~\left( \func{mod}%
12\right) .
\end{equation*}%
b) Let $p\equiv 7~\left( \func{mod}12\right) ~$be a prime. Then 
\begin{equation*}
b\equiv 4~\left( \func{mod}12\right) ~iff~~N_{p,a}\equiv 4~\left( \func{mod}%
12\right)
\end{equation*}%
and%
\begin{equation*}
b\equiv 8~\left( \func{mod}12\right) ~iff~~N_{p,a}\equiv 0~\left( \func{mod}%
12\right) .
\end{equation*}
\end{theorem}

\begin{proof}
a) Let $p\equiv 1~\left( \func{mod}12\right) $ be a prime. Then we can write
this as $p=1+12n,~n\in 
\mathbb{Z}
.$ Also $b\equiv 2~\left( \func{mod}12\right) ~$can be stated as $b=2+12m,$ $%
m\in 
\mathbb{Z}
.~$By substituting these, we get 
\begin{equation*}
b\equiv 2~\left( \func{mod}12\right) \Longleftrightarrow N_{p,a}=p+1-b\ \ \ 
\end{equation*}%
and hence $N_{p,a}$ $=1+12n+1-\left( 2+12m\right) =12\left( n-m\right) $ and
this is only valid when $N_{p,a}\equiv 0~\left( \func{mod}12\right) .~$%
Similarly,%
\begin{equation*}
b\equiv 10~\left( \func{mod}12\right) \Longleftrightarrow
N_{p,a}=p+1-b=1+12n+1-\left( 10+12m\right)
\end{equation*}%
and therefore\ $N_{p,a}\ =-8+12\left( n-m\right) ~$and this means that $%
N_{p,a}\equiv 4~\left( \func{mod}12\right) .~$Part b) is proved in a similar
fashion.
\end{proof}

\begin{theorem}
Let $p\equiv 1~\left( \func{mod}6\right) $ be a prime. Then $b$ is not
divisible by $6$.
\end{theorem}

\begin{proof}
Let us consider the curve $y^{2}=x^{3}+1.$ It has a point of order 6.
Therefore its reduction modulo p has also a point of order 6. Therefore%
\begin{equation*}
b\equiv p+1-N_{p,a}\equiv 2-0\equiv 2\text{ \ }\left( \func{mod}6\right) .
\end{equation*}

The other possibility for the curve is $y^{2}=x^{3}+a^{3}$ with $a$ is a
quadratic non-residue. It is the quadratic twist of the other curve, so has $%
b\equiv -2~\ \left( \func{mod}6\right) .$ Therefore in both cases $b$ is
non-zero $\func{mod}6.$
\end{proof}

\begin{corollary}
Let $p\equiv 1~\left( \func{mod}6\right) $ be a prime. Then $N_{p,a}\equiv 0$
$~or~N_{p,a}\equiv 4\left( \func{mod}6\right) .$
\end{corollary}

Also one obtains the following result:

\begin{corollary}
If $p\equiv 1~\left( \func{mod}12\right) $ is a prime, then $b\equiv \mp
2~\left( \func{mod}12\right) $ and if $p\equiv 7~\left( \func{mod}12\right) $
is a prime, then $b\equiv \mp 4~\left( \func{mod}12\right) .$
\end{corollary}

We now have the following result about the number of points on curves $%
\left( 1\right) .$

\begin{theorem}
Let $p\equiv 1~\left( \func{mod}6\right) $ be a prime. Then

a) If $t\equiv 2$ $\left( \func{mod}6\right) ,$ then $\left( 1\right) $ has $%
b=t$ and $N_{p,a}\equiv 0~\left( 6\right) ,$ and its twist has $b=-t$ and $%
N_{p,a}\equiv 4~\left( \func{mod}6\right) .$

b) If $t\equiv 4~\left( \func{mod}6\right) ,$ then $\left( 1\right) $ has $%
b=t$ and $N_{p,a}\equiv 4~\left( \func{mod}6\right) ,$ and its twist has $%
b=-t$ and $N_{p,a}\equiv 0~\left( \func{mod}6\right) .$
\end{theorem}

\begin{proof}
Let $p\equiv 1~\left( \func{mod}6\right) $ be a prime. Let us put $%
p=1+6n,~n\in 
\mathbb{Z}
.$ Let $t=2~\left( \func{mod}6\right) .$ If $b=t,$ then $b\equiv 2~\left( 
\func{mod}6\right) $ and now put $b=2+6m,~m\in 
\mathbb{Z}
.$ Therefore 
\begin{eqnarray*}
N_{p,a} &=&p+1-b=6n+1+1-2-6m \\
&=&6\left( n-m\right)
\end{eqnarray*}%
implying $N_{p,a}\equiv 0~\left( \func{mod}6\right) .$

The other parts can be proven similarly.
\end{proof}

We then immediately have the following result concerning the elements of
order 3:

\begin{corollary}
a) Let $p\equiv 1~\left( \func{mod}12\right) $ be a prime. If $t\equiv
2~\left( \func{mod}12\right) ,$ then $\left( 1\right) $ has $b=t$ and $%
N_{p,a}\equiv 0~\left( \func{mod}12\right) ~$and $E\left( \mathbb{F}%
_{p}\right) $ has elements of order $3$. Its twist has $b=-t$ and $%
N_{p,a}\equiv 4~\left( \func{mod}12\right) $ implying that there are no
elements of order $3.$

If $t\equiv 10~\left( \func{mod}12\right) ,$ then $\left( 1\right) $ has $%
b=t $ and $N_{p,a}\equiv 4~\left( \func{mod}12\right) $ and $E\left( \mathbb{%
F}_{p}\right) $ has no elements of order $3,$ while its twist has $b=-t$ $\ $%
and $N_{p,a}\equiv 0~\left( \func{mod}12\right) $ implying that the group
has elements of order $3.$

b) Let $p\equiv 7~\left( \func{mod}12\right) $ be a prime. If $t\equiv
4~\left( \func{mod}12\right) ,$ then $\left( 1\right) $ has $b=t$ and $%
N_{p,a}\equiv 4~\left( \func{mod}12\right) $ and therefore has no points of
order $3,$ while its twist has $b=-t$ and $N_{p,a}\equiv 0~\left( \func{mod}%
12\right) $ having elements of order $3.$

If $t\equiv 8~\left( \func{mod}12\right) $, then $\left( 1\right) $ has $b=t$
and $N_{p,a}\equiv 0~\left( \func{mod}12\right) $ implying it has elements
of order $3$ while its twist has $b=-t$ and $N_{p,a}\equiv 4~\left( \func{mod%
}12\right) $ having no such elements.
\end{corollary}

The elements of order $3$ are important in the classification of these
elliptic curves modulo $p$. We now show that their number is either $2$ or $%
8 $:

\begin{theorem}
Let $p\equiv 1~\left( \func{mod}6\right) $ be a prime. If $N_{p,a}\equiv
0~\left( \func{mod}6\right) ,$ then there are $2$ or $8$ points of order $3.$
\end{theorem}

\begin{proof}
By $\cite{Schoof},~$there are at most 9 points together with the point at
infinity $o,$ forming a subgroup which is either trivial, cyclic of order $3$
or the direct product of two cyclic groups of order $3$. As we want to
determine the number of elements of order\ $3,$ this group cannot be
trivial. Then it is $C_{3}$ or $C_{3}\times C_{3}$ and it is well-known that
it contains $2$ or $8$ elements of order $3,$ respectively.
\end{proof}

In fact, if we let $E\left( \mathbb{F}_{p}\right) \cong C_{n}\times C_{nm},$
then when $3$ divides $n,~E\left( \mathbb{F}_{p}\right) $ has $8$\ points of
order $3,$ and when not, it has $2$ points of order $3.$

We are now going to give one of the main results in Theorem 13. We first
need the following results:

\begin{corollary}
Let $p$ be a prime. Then for only $x=0$ among all values of $x$ in $\mathbb{F%
}_{p}$, $x^{3}+1$ takes the value $1.$
\end{corollary}

\begin{proof}
It is clear that $x=0$ satisfies the condition. The fact that no other value
of $x$ satisfies $x^{3}+1=1~$is clear from the fact $p$ is prime.
\end{proof}

\begin{theorem}
Let $p\equiv 1~\left( \func{mod}6\right) $ be a prime. There are $3$ values
of$\ x$ between $1$ and $p$ so that $x^{3}+1\equiv 0~\left( \func{mod}%
p\right) .$
\end{theorem}

\begin{proof}
It is obvious that $x^{3}\equiv a~\left( \func{mod}p\right) $ has three
solutions in $\mathbb{F}_{p}$ for every $a\neq 0.$ For $a=-1,$ the proof
follows.
\end{proof}

\begin{theorem}
Let $\ p\equiv 1~\left( \func{mod}6\right) $ be a prime. Then 
\begin{equation*}
\underset{x\in \mathbb{F}_{p}}{\sum }\chi \left( x^{3}+1\right) \equiv
4~\left( \func{mod}6\right) .
\end{equation*}
\end{theorem}

\begin{proof}
For each $x\in \mathbb{F}_{p},$ calculate the $p$ values of $x^{3}+1.$ By
Theorem 10, one of these values is $1$. By Theorem 11, three of them are 0.
The rest $p-4$ values of $x^{3}+1$\ are grouped into $\frac{p-4}{3}$
triples. As $p\equiv 1~\left( \func{mod}6\right) ,$ $\frac{p-4}{3}$ is odd.
Indeed, let us write $p=1+6k,~k\in 
\mathbb{Z}
.$ Then $\frac{p-4}{3}=2k-1.$ Let us suppose that out of these triples, s
triples are in $Q_{p}$ and $2k-1-s$ are in $Q_{p}^{^{\prime }}.$ If a triple
is in $Q_{p},$ then it adds $+3$ to the sum $\underset{x\in \mathbb{F}_{p}}{%
\sum }\chi \left( x^{3}+1\right) ,~$and if it is in $Q_{p}^{^{\prime }},$ $%
-3 $ is added. Therefore%
\begin{eqnarray*}
\underset{x\in \mathbb{F}_{p}}{\sum }\chi \left( x^{3}+1\right)
&=&1+3.0+s.\left( +3\right) +\left( 2k-1-s\right) .\left( -3\right) \\
&=&6\left( s-k\right) +4
\end{eqnarray*}%
implying the result.
\end{proof}

\begin{theorem}
Let $p\equiv 1~\left( \func{mod}6\right) $ be a prime. $a\in Q_{p}$ \ iff $%
N_{p,a}\equiv 0~\left( \func{mod}6\right) .$
\end{theorem}

\begin{proof}
It is well-known that 
\begin{equation*}
N_{p,a}=p+1+\underset{x\in \mathbb{F}_{p}}{\sum }\chi \left(
x^{3}+a^{3}\right)
\end{equation*}%
By putting $p=1+6n$ for $n\in 
\mathbb{Z}
,$ we get $N_{p,a}=6n+2+\underset{x\in \mathbb{F}_{p}}{\sum }\chi \left(
x^{3}+a^{3}\right) .$ Now as $\chi \left( a\right) =1,$ and as the set of
the values of $x^{3}$ is the same as the set of the values of $a^{3}x^{3},$
we can write%
\begin{eqnarray*}
\underset{x\in \mathbb{F}_{p}}{\sum }\chi \left( x^{3}+a^{3}\right) &=&%
\underset{x\in \mathbb{F}_{p}}{\sum }\chi \left( a^{3}x^{3}+a^{3}\right) \\
&=&\underset{x\in \mathbb{F}_{p}}{\sum }\chi \left( a^{3}\right) \chi \left(
x^{3}+1\right) \\
&=&\underset{x\in \mathbb{F}_{p}}{\sum }\chi \left( x^{3}+1\right) ,
\end{eqnarray*}%
and by Theorem 12, this sum is congruent to 4 modulo 6. Hence, by putting $%
\underset{x\in \mathbb{F}_{p}}{\sum }\chi \left( x^{3}+a^{3}\right)
=4+6r,~r\in 
\mathbb{Z}
,$ we get $N_{p,a}=6n+2+4+6r$ implying $N_{p,a}\equiv 0~\left( \func{mod}%
6\right) .$
\end{proof}

\begin{corollary}
Let $p\equiv 1~\left( \func{mod}6\right) ~$be a prime. If $N_{p,a}\equiv
0~\left( \func{mod}6\right) ,$ then $b\equiv 2~\left( \func{mod}6\right) .$
\end{corollary}

\begin{proof}
As $N_{p,a}=p+1-b=p+1+\underset{x\in \mathbb{F}_{p}}{\sum }\chi \left(
x^{3}+a^{3}\right) ,$ we know that $b=-\underset{x\in \mathbb{F}_{p}}{\sum }%
\chi \left( x^{3}+a^{3}\right) .$ By Theorem 12, the result follows.
\end{proof}

Similarly we have

\begin{theorem}
Let $p\equiv 1~\left( \func{mod}6\right) ~$be a prime. Then $a\in
Q_{p}^{^{\prime }}$ iff $N\equiv 4~\left( \func{mod}6\right) .$
\end{theorem}

\begin{corollary}
Let $p\equiv 1~\left( \func{mod}6\right) ~$be a prime. Let $E$ be the curve
given by $\left( 1\right) .$ Then

a) $a\in Q_{p}$ iff $E\left( \mathbb{F}_{p}\right) $ has $2$ or$\ 8$
elements of order $3.$

b) $a\in Q_{p}^{^{\prime }}$ iff $E\left( \mathbb{F}_{p}\right) $ has no
elements of order $3.$
\end{corollary}

\begin{proof}
It is clear from Corollary 8 and Theorem 13.
\end{proof}

\section{Bachet Elliptic Curves having a group of the form $C_{n}\times
C_{n}.$}

Now we shall consider the case where the Bachet elliptic curves have a group
isomorphic to $C_{n}\times C_{n}$ for same n. This is only possible when $%
p\equiv 1~\left( \func{mod}6\right) ,~$as otherwise when $p\equiv 5~\left( 
\func{mod}6\right) ,~E\left( \mathbb{F}_{p}\right) $ is isomorphic to the
cyclic group $C_{p+1}.$ We shall consider a result of Washington and refine
it:

\begin{theorem}
$\cite{Washington}~$Let $E$ be an elliptic curve over $\mathbb{F}_{q}$ where
q is a prime power and suppose $E\left( \mathbb{F}_{q}\right) \cong 
\mathbb{Z}
_{n}\times 
\mathbb{Z}
_{n}$ for some integer n. Then either $q=n^{2}+1,~q=n^{2}\mp n+1\ $or $%
q=\left( n\mp 1\right) ^{2}.$
\end{theorem}

Now we give a more specific result for Bachet elliptic curves given by $%
\left( 1\right) $ over $\mathbb{F}_{q}:$

\begin{theorem}
Let $E$ be the elliptic curve in $\left( 1\right) .$ Suppose 
\begin{equation*}
E\left( \mathbb{F}_{p}\right) \cong 
\mathbb{Z}
_{n}\times 
\mathbb{Z}
_{n}.
\end{equation*}%
Then $p\equiv 7~\left( \func{mod}12\right) $ and $p=n^{2}\mp n+1.$
\end{theorem}

\begin{proof}
By Theorem 17, there are three possibilities $p=n^{2}+1,~p=n^{2}\mp n+1~$or $%
p=n^{2}\mp 2n+1.$ The latter one is immediately rules out as $p$ cannot be a
square. We need only show that $p$ cannot be equal to $n^{2}+1.$

If $p=n^{2}+1,~$than $n^{2}=p-1$ and hence $p-1$ is in $Q_{p}.$ But it is
known that $p-1$ could be in $Q_{p}$ only when $p\equiv 1,5~\left( \func{mod}%
12\right) $ is prime. Therefore the result follows.
\end{proof}

\begin{tabular}{l}
Musa Demirci, G\"{o}khan Soydan, \.{I}smail Naci Cang\"{u}l \\ 
Department of Mathematics \\ 
Uluda\u{g} University \\ 
16059 Bursa, TURKEY \\ 
mdemirci@uludag.edu.tr, gsoydan@uludag.edu.tr, cangul@uludag.edu.tr%
\end{tabular}

\begin{tabular}{l}
Nazl\i\ Y\i ld\i z \.{I}kikarde\c{s} \\ 
Department of Mathematics \\ 
Bal\i kesir University \\ 
Bal\i kesir, TURKEY \\ 
nyildiz@balikesir.edu.tr%
\end{tabular}


\begin{thebibliography}{9}
\bibitem{Desoca} Demirci, M. \& Soydan, G. \& Cang\"{u}l, I. N., \emph{%
Rational points on the\ elliptic curves \ }$y^{2}=x^{3}+a^{3}\,(\limfunc{mod}%
\,p)$\emph{\ in }$\mathbb{F}_{p}$ w\emph{here p}$\equiv 1(\func{mod}6)$\emph{%
\ is prime, }Rocky Mountain Journal Of Mathematics,Vol.37, No:5 (2007),
1467-1475.

\bibitem{SDYC} Soydan, G. \& Demirci, M. \& Ikikarde\c{s}, N. Y. \& Cang\"{u}%
l, I. N., \emph{Rational points on the\ elliptic curves \ }$%
y^{2}=x^{3}+a^{3}\,(\limfunc{mod}\,p)$\emph{\ in }$\mathbb{F}_{p}$ w\emph{%
here p}$\equiv 5$ $(\func{mod}6)$\emph{\ is prime, }International Journal Of
Mathematics Sciences, Vol.1, No:4 (2007), 247-250.

\bibitem{Washington} Washington, L. C., \emph{Elliptic Curves, Number Theory
and Cryptography, }Chapman\&Hall/CRC, 2003

\bibitem{Schmitt} Schmitt, S., Zimmer, H. G., \emph{Elliptic Curves,a
computational approach, }Walter de Gruyter, 2003

\bibitem{Schoof} Schoof, R., \emph{Nonsingular Plane Cubic Curves over
Finite Fields, }J.Comb. Thry., Ser. A, Vol. 46, No: 2, 1987
\end{thebibliography}
\end{document}